\documentclass[10pt]{tran-l}
\usepackage{graphicx}
\usepackage{setspace}
\usepackage{marginnote}
\usepackage{amsthm}
\usepackage{comment}
\usepackage{mathrsfs}
\usepackage{hyperref}
\usepackage{amscd,todonotes}
\hypersetup{linkcolor=red,citecolor=black}
\usepackage{soul}

\newtheorem{theorem}{Theorem}[section]
\newtheorem{theo}{Theorem}

\newtheorem{lem}[theorem]{Lemma}

\newtheorem{cor}[theorem]{Corollary}
\newtheorem*{conj}{Question}

\theoremstyle{definition}
\newtheorem{df}[theorem]{Definition}
\newtheorem{ex}[theorem]{Example}

\newtheorem{example}{Example}

\theoremstyle{remark}

\numberwithin{equation}{section}

\newcommand{\p}{\mathbf{P}}
\newcommand{\T}{\tau}
\newcommand{\x}{\boldsymbol{x}}

\newcommand{\F}{\mathscr{F}}
\newcommand{\g}{\boldsymbol{g}}

\newcommand{\m}{\boldsymbol{\mu}}
\newcommand{\pa}{\partial}
\newcommand{\ta}{\boldsymbol{\tau}}
\newcommand{\RR}{\mathbb{R}}

\newcommand{\ZZ}{\mathbb{Z}}
\newcommand{\NN}{\mathbb{N}}

\newcommand{\Z}{\ZZ}

\DeclareMathOperator{\supp}{supp}

\setstcolor{red}

\begin{document}

\title{Positive harmonic functions of transformed random walks}

\author{Behrang Forghani}
\address{Department of Mathematics, University of Connecticut, USA}
\email{behrang.forghani@uconn.edu}
\author{Keivan Mallahi-Karai}
\address{Department of Mathematics and Logistic, Jacobs University of Bremen, Germany}
\email{k.mallahikarai@jacobs-university.de}

\keywords{Random walk, Harmonic functions, Martin boundary}

\begin{abstract}
In this paper, we will study the behavior of the space of positive harmonic functions  associated with the random walk on a discrete group under the change of probability measure by a randomized stopping time. We show that this space remains unchanged after applying a bounded randomized stopping time.
\end{abstract}

\maketitle
\section*{Introduction}
Let $\mu$ be a  probability measure on a discrete group (or semigroup) $\Gamma$. We say $\mu$ is generating if the semigroup generated by the support of $\mu$ is $\Gamma$. Let $h_1,h_2, \dots $ be a sequence of $\Gamma$-valued independent $\mu$-distributed random variables. The Markov chain $(x_n)_{n\geq0}$ defined by $x_0=e$, the identity element of $\Gamma$, and $x_n:=h_1 \cdots h_n, n \ge 1$ is called the random walk on $\Gamma$ associated with the probability measure $\mu$. 

Various qualitative and quantitative objects have been developed to capture the behavior of $(x_n)_{n\geq0}$ at infinity. One of these qualitative objects is the Martin boundary,  first introduced by Martin \cite{Martin41}, and later investigated thoroughly by Doob \cite{Doob59}, Hunt \cite{Hunt60} and Dynkin \cite{Dynkin69}.   A concise description of the Martin boundary can be given as follows. The Green function associated with $\mu$ is defined by $G_{\mu}(x,y)= \sum_{n=0}^{ \infty} \mu^{*n}(x^{-1}y)$. Here $\mu^{*n}$ denotes the $n$-th convolution power of $\mu$. Assume that the random walk $(\Gamma
,\mu)$ is transient, that is, the Green functions are finite for all $x$ and $y$ in $\Gamma$. The Martin compactification of $\Gamma$, denoted by $\hat{\Gamma}$, is the unique compactification of $\Gamma$ with the property that for all $x$ in $\Gamma$ the Martin kernels $K_{\mu}(x,.)= G_{\mu}(x,.)/G_{\mu}(e,.)$ can be extended to continuous functions on $\hat{\Gamma}$, and that they separate points in $\hat{\Gamma}\backslash\Gamma$. The compact set $\partial_\mu\Gamma:=\hat{\Gamma}\backslash\Gamma$ is called the Martin boundary of the random walk $(G,\mu)$.  Recall that a positive $\mu$--harmonic function is a  positive function on $\Gamma$ satisfying $f(g)= \sum_{h \in \Gamma }f(gh)\mu(h) $ for all $g \in \Gamma $. Denote the space of positive harmonic functions associated with $\Gamma$ and $\mu$ by  $H^+(\Gamma ,\mu)$. Every positive harmonic function has a representation as an integral of Martin kernels with respect to some measure on the Martin boundary. 
A measure--theoretical counterpart of the Martin boundary is the Poisson boundary,
which can be viewed as the subset of the Martin boundary that supports the probability measure defining the constant function $1$. 

Let us now fix a discrete group $\Gamma$ equipped with a generating probability measure $\mu$. Giving a concrete description of the Martin boundary (or the Poisson boundary) of the random walk $(\Gamma,\mu)$ is, in general, a difficult task and a priori depends on the probability measure $\mu$. In fact, there are examples of very different Martin boundaries arising from different probability measures. For example, one can equip $\Z^3$ with generating measures $\mu_1$ and $\mu_2$ such that the Martin boundary of the random walk $(\Z^3,\mu_1)$ is trivial, while
the Martin boundary of the random walk $(\Z^3,\mu_2)$ is homeomorphic to a sphere \cite{Ney-Spitzer}. 
 On the other hand, there are also situations in which $\partial_\mu\Gamma$ does not depend on the choice of the probability measure $\mu$.  
This work has been motivated by the desire to answer the following question:
\begin{conj}
Characterize probability measures $\mu'$ on $\Gamma$ such that the Martin boundary (Poisson boundary, respectively) of the random walk $(\Gamma,\mu')$ coincides with the Martin boundary (Poisson boundary, respectively) of the random walk $(\Gamma,\mu)$.
\end{conj}
Ancona \cite{Ancona88} showed that when $\Gamma$ is a hyperbolic group and $\mu$ finitely supported, then the Martin boundary of $\partial_\mu\Gamma$  coincides with the hyperbolic boundary of $\Gamma$. Recently, Gou\"{e}zel generalized this theorem of Ancona to the case that the probability measure $\mu$ has finite exponential moment, see \cite{Gouezel2015}. The preceding question has been studied for the measure--theoretical counterpart of the Martin boundary, i.e., the Poisson boundary. For instance, Furstenberg showed that if $H$ is a recurrent subgroup of $\Gamma$, then the Poisson boundary of $(\Gamma,\mu)$ coincides with the one of the induced random walk on $H$ \cite{Fu70}. Recently,  Kaimanovich and the first author introduced a general method of constructing random walks which  preserves the Poisson boundary of $(\Gamma,\mu)$, and it subsumes all the already known examples \cite{BK2013, Behrang2016}. Their method is based on employing a randomized stopping time, which can be viewed as a randomly drawn stopping time. Let $\tau$ be a randomized stopping time which is almost surely finite. Since random walks on groups are time and space homogeneous, iterations of $\tau$ lead to a sequence of randomized stopping times $(\T_1=\T, \T_2,\cdots)$. Almost every sample path $\x=(x_0,x_1,x_2,\cdots)$ of the random walk $(\Gamma ,\mu)$ can then be transformed to a random subsequence $(x_0,x_{\T_1(\x)},x_{\T_2(\x)},\cdots).$ These subsequences are indeed sample paths of a random walk with respect to the {\it transformed} probability measure $\mu_\T$, which assigns to an element $g \in \Gamma$ the probability that sample paths of the original random walk $(\Gamma ,\mu)$ stop for the first time at $g$. Under a logarithmic moment condition, Kaimanovich and the first author show that a randomized stopping time preserves the Poisson boundary. Their proof relies on the entropy criterion, which is a quantitative method to verify whether or not a $\Gamma$--space can be identified as the Poisson boundary \cite{BK2013}. This method is no longer applicable to Martin boundaries. In fact, there are no known substitutes for the entropy criterion to single out the Martin boundary. Given that the Martin boundary  characterizes all positive harmonic functions, the stability of the space of positive harmonic functions under the modification of the random walk may be viewed as an approximation of the phenomenon in question.

Our primary objective in this paper is to study the invariance of the Martin boundary and the space of positive harmonic functions under the change of underlying measure using a randomized stopping time. We will illustrate the delicateness of the situation for the Martin boundary by giving two examples which show that if we impose no restrictions whatsoever on $\tau$, the Martin boundary and the space of positive harmonic functions may change as we pass from $\mu$ to $\mu_{\tau}$.

\begin{example}\label{counterex:semigroup}
Consider the free semigroup $\F_2$ generated by the generators $a$ and $b$, equipped with the probability measure $ \mu = \frac{1}{2} \delta_a+ 
\frac{1}{2} \delta_b$. One can show that (see Example~\ref{wil} and Example~\ref{ex : convex}) there exists a randomized stopping time $\T$ such that
$$
\mu_\T=\frac12\mu+ \sum_{n=0}^{ \infty} \frac{1}{2^{n+2}} \delta_{b^na}.
$$ 

Let $f$ be defined by $f(b^n)=2^n$ for $n \ge 0$ and $f(x)=0$ for all other values of $x$. It is easy to see that $f$ is $\mu$-harmonic, but not $\mu_\T$-harmonic. 
\end{example} 

Here is a somewhat more interesting example involving groups. 

\begin{example}\label{counterex:group}
Let $e_1,e_2, e_3$ be the canonical basis for $\Gamma= \ZZ^3$. 
Consider the random walk $(\ZZ^3,\mu)$ where 
\[ \mu( \pm e_1)=\mu( \pm e_3)= \frac{1}{8}, \quad \mu(e_2)= \frac{1}{8}, \quad \mu(-e_2)= \frac{3}{8}. \]
Let $\tau$ be the stopping time defined to be the 
least $k$ such that $x_k$ lands in the cyclic subgroup $A$ generated by $e_2$.

Since $\mu$ is not centered, it follows from a theorem of Ney and Spitzer \cite{Ney-Spitzer} (see also
\cite{Woess-Book}, Theorem 25.15) that the Martin boundary of $\mu$ is homeomorphic to the two-dimensional sphere. On the other hand, it is clear that for $x,y \in A$, $K_{\mu}(x,y)=K_{\mu_{\T}}(x,y)$. In particular, the sequences
$a_n=ne_2$ and $a_n=-ne_2$ are both convergent in the Martin boundary of the random walk $(\ZZ^3,\mu_\T)$. This implies that the Martin  boundary associated with $\mu_\T$ contains at most two points. 
\end{example}

As it can be observed in the last example, the probability measure $\mu_{\T}$ is not necessarily generating. In order to guarantee that the transformed probability measure is generating,  we consider the convex combination 
$$\mu_{\T,t}:= t\mu+(1-t)\mu_\T,$$ 
where $\T$ is a randomized stopping time and $t \in (0,1]$ is a parameter. The contribution from $\mu$ will guarantee that the semigroup generated by the support of $\mu$ coincides with the semigroup generated by the support of $\mu_{\T,t}$. It is not hard to see that, in fact, $\mu_{\T,t}$ itself coincides with $\mu_{\T'}$ for an appropriate randomized stopping time $\T'$; see Example~\ref{ex : convex}. 

\begin{theo}\label{thm:bnd}
Let $\mu$ be a generating probability measure on a discrete group $\Gamma$, and let $\T$ be a bounded randomized stopping time. Then $H^+(\Gamma ,\mu)=H^+(\Gamma ,\mu_{\T,t})$, where $\mu_{\T,t}=t\mu+(1-t)\mu_\T$ for some $0<t\leq1$.
\end{theo}

Note that the almost surely finite stopping times constructed in Examples \ref{counterex:semigroup} and \ref{counterex:group} are not bounded and hence do not fall under the purview of Theorem \ref{thm:bnd}. 

Theorem~\ref{thm:bnd} implies that if $\rho_1,\cdots,\rho_k$ are bounded randomized stopping times and  $\mu'$ is a finite convex combination of the probability measures $\mu, \mu_{\rho_1},\cdots,\mu_{\rho_n}$, that is $\mu'=a_1\mu+a_2\mu_{\rho_1}+\cdots+a_{n+1}\mu_{\rho_n}$ where $a_i>0$ for $i=1,\cdots, n+1$ and $\sum_{i=1}^{n+1}a_i=1$, then  $H^+(\Gamma ,\mu)=H^+(\Gamma ,\mu')$.

In particular, let $\rho_i$ be the constant function $i$ for $i=1,\cdots, n$. Then $\mu'$ is a finite convex combination of convolutions of $\mu$ and $H^+(\Gamma ,\mu)=H^+(\Gamma ,\mu')$. Nevertheless, Kaimanovich \cite{K83} has shown  the probability measure $\mu'$ can be an infinite convex combination of convolutions of $\mu$.
 We will generalize the aforementioned result to the probability measures that commute under convolution. 
\begin{theo}
Consider a set of probability measures ${\theta_i}$ on a discrete group $\Gamma$ such that $\theta_1*\theta_i=\theta_i*\theta_1$ for all $i\geq1$. Let  $\theta =\sum_{i}a_i\theta_i$ where $a_i$ is nonnegative for all $i\geq 1$ and $\sum_{i\geq 1} a_i=1$. If $\theta_1$ is generating and $a_1>0$ such that $H^+(\Gamma,\theta_1)\subset H^+(\Gamma,\theta_i)$ for $i\geq 1$, then
$H^+(\Gamma ,\theta)= H^+(\Gamma ,\theta_1)$. 
\end{theo} 
As it was indicated above, the Martin boundary of a random walk $(\Gamma ,\mu)$ could be very different from that of $(\Gamma ,\mu'=\sum_{n\geq1}a_n\mu^{*n})$. In the case of simple random walk on a free group, Cartwright and Sawyer have provided a necessary and sufficient condition for the equality of the Martin boundaries associated with $\mu$ and $\mu'$. In particular, they show that there exists a convex combination of convolutions of the simple random walk whose Martin boundary is different from that of the simple random walk \cite{Cartwright-Sawyer91}.

Let us present the main ideas of the proof of Theorem \ref{thm:bnd}. The first step is to lift harmonic functions on $\Gamma$ to harmonic functions on a free semigroup covering $\Gamma$. We will then give a concrete geometric description of the Martin boundary of the random walk and the associated space of minimal harmonic functions on this free semigroup. 
\begin{theo}
Let $\F$ be a finitely or infinitely generated free semigroup generated by a generating set $W$. Let $\mu$ be a  probability measure supported on $W$. For any bounded randomized stopping time $\T$ and $0<t\leq 1$, the Martin boundary of $(\F,\mu_{t,\T})$ coincides with the geometric boundary of the free semigroup $\F$. 
\end{theo}
This can be used to show that if the randomized stopping time $\tau$ is bounded, then the space of positive harmonic functions with respect to the random walk $\mu_{\T,t}$ is left unchanged. At the end, we will use the classical result due to Dynkin and Maljutov \cite{Dynkin-Maljutov61} to deduce the invariance of the space of positive harmonic functions for the random walks on $\Gamma$.

The idea of lifting harmonic functions to the free semigroup originates from \cite{BK2013} where it was used to show that randomized stopping times with first logarithmic moment do not affect the Poisson boundary.  One should note that coincidence of the Poisson boundaries means the coincidence of almost all (with respect to the harmonic measure on the Poisson boundary) minimal harmonic functions associated to points of the Poisson boundary. Theorem~\ref{thm:bnd} concerns {\it all} (instead of almost all) minimal harmonic functions; however, the price to pay is to assume that our randomized stopping time is {\it bounded} (rather than having a finite logarithmic moment). 
Traditionally, finiteness of support is the standard condition for dealing with the Martin boundary; the only notable exception in this regard, as it was mentioned earlier, is in the recent work of Gou\"{e}zel on random walks on hyperbolic groups where the finiteness of support is replaced by having finite exponential moment \cite{Gouezel2015}.

This paper is organized as follows. In Section~1, we will set the preliminaries and define the Martin compactification. Section~2 is devoted to the construction of transformed random walks.
In Section~3, we will consider a random walk associated with a probability measure which is distributed on free generators of a free semigroup. In this setup, we describe the minimal harmonic functions and the Martin boundary associated to  probability measures $\mu_{\T'}$ and $\mu_{\T,t}$ where $\T'$ are a bounded stopping times and  a bounded randomized stopping time, respectively and $0<t\leq 1$. Finally, the last section will be about the main results for discrete groups.

\section{Random walks on groups}\label{RW}
Throughout this paper, $\Gamma$ denotes a discrete group, endowed with a probability measure $\mu$. Assume the semigroup generated by $\mu$ is $\Gamma$, in other words, $\mu$ is generating. 
Define the random walk $(\Gamma ,\mu)$ as the Markov chain with the state space $\Gamma$ and the
transition probabilities given by
$$
p(g_1,g_2)=\mu(g^{-1}_1g_2), \quad g_1, g_2 \in \Gamma .
$$
The \emph{space of increments} of the random walk $(\Gamma ,\mu)$ is denoted by $(\Gamma ^{\mathbb N},\bigotimes\limits_{1}^{\infty}\mu)$, where $\bigotimes$ denotes the product measure. One can also view the space of increments as the set of sequences $(h_n)_{n\geq 1}$ of independent and identically $\mu$-distributed random variables.
The space of sample paths of the random walk $(\Gamma ,\mu)$ is the probability
space $(\Gamma ^{\Bbb Z+},\p)$,  where $\Gamma^{\Bbb Z+}=\{e\}\times \Gamma^{\Bbb N}$. Note that we only consider  random walks issued from the identity element $e$, and the probability measure $\p$  is the push-forward of the measure
$\bigotimes_{1}^{\infty}\mu$ under the map
$$
(h_1,h_2,\cdots)\mapsto \x=(e,x_1,x_2,\cdots),
$$
where $x_n=h_1\cdots h_n$ denotes the position of the random walk $(\Gamma ,\mu)$ at time $n$.

\subsection{The Martin boundary} In this section, we will briefly discuss the notion of Martin 
boundary of a probability measure on a discrete group or semigroup. 
A more detailed discussion can be found in \cite{Dynkin69,Sawyer1997, Woess09}.\\
Let $X$ be a countable infinite set. A compactification of $X$ is a compact Hausdorff space $ \hat{\rule{0ex}{1.7ex}\mkern-3mu X}$ which contains $X$ as a dense subset, such that the induced topology from $ \hat{\rule{0ex}{1.7ex}\mkern-3mu X}$ to $X$ is discrete.  Two compactifications of $X$ are deemed to be equal if the identity map $X \to X$ extends to a homeomorphism between them. 

Assume now that $\Gamma$ is a discrete group and $\mu$ is a generating probability measure on $\Gamma$. 
For $x,y \in \Gamma$, let $F_{\mu}(x,y)$ denote the probability that the trajectory of the random walk, when started at $x$, ever hits $y$. The Markov chain $(\Gamma ,\mu)$ is called \emph{transient} when $F_{\mu}(x,y)<1$ for all $x,y \in \Gamma$. One can also define transience using Green functions. Recall that the Green function of $\mu$ is defined by 
$$
G_{\mu}(x,y)=\sum_{n\geq 0}\mu^{*n}(x^{-1}y).
$$ 
One can show that transience of $(\Gamma, \mu)$ is equivalent to the condition $G_{\mu}(x,y)< \infty$ for all $x,y \in \Gamma $. From now on, all Markov chains considered in this section are assumed to be 
transient. This allows one to define the \emph{Martin kernel} by
 $$
 K_{\mu}(x,y)=\frac{G_{\mu}(x,y)}{G_{\mu}(e,y)}.
 $$

In order to define the Martin compactification, let us equip $\Gamma$ with the following metric: 
\begin{df}
Fix a bijection $N: \Gamma \to \NN $. The compactification $\hat{\Gamma}$ with respect to the metric
$$
d(y,z)=|2^{-N(y)}-2^{-N(z)}|+\sum_x|K_\mu(x,y)-K_\mu(x,z)|F_\mu(e,x)2^{-N(x)}
$$
is called the \emph{Martin compactification} of the random walk $(\Gamma ,\mu)$. Furthermore, 
$\partial_\mu\Gamma:= \hat{\Gamma} \setminus \Gamma$ is called the \emph{Martin boundary} of the random walk. 
\end{df}

One can show that  $\hat\Gamma$ is the Martin compactification of $(\Gamma,\mu)$ if and only if   for a fixed $x \in \Gamma $, the Martin kernel $K_\mu(x,.)$ extends uniquely to a continuous function defined on  $\hat\Gamma$  in such a way that that for boundary points $\eta$ and $\xi$, we have $K_\mu( \cdot ,\eta)=K_\mu( \cdot ,\xi)$ if and only if $\eta=\xi$.   
 
\subsection{Harmonic functions}
A function $f: \Gamma \to \RR$ is called $\mu$--\emph{harmonic} whenever for every $g$  in $\Gamma$,  the series defining the expression 
$$P^\mu f(g)=\sum_{h\in \Gamma }f(gh)\mu(h)$$
is absolutely convergent and
$P^\mu f=f.$
Denote the set of all positive $\mu$-harmonic functions by $H^+(\Gamma ,\mu)$. 

Harmonic functions are intimately connected to the properties of random walks on groups. For instance, the classes of bounded and positive harmonic functions with respect to a probability measure are known to be related to the Poisson and Martin boundaries associated to that measure. In studying positive harmonic functions, a key role is played by \emph{minimal harmonic functions}. A positive harmonic function $f$ is called minimal, if $f(e)=1$ and for any decomposition $f=f_1+f_2$ of $f$ into a sum of positive harmonic functions $f_1$ and $f_2$, both $f_1$ and $f_2$ are proportional to $f$.
\medskip

The subset of the Martin boundary consisting of points $\xi$ where $K_\mu( \cdot ,\xi)$ is a minimal $\mu$--harmonic is called the minimal Martin boundary, and is denoted by 
$\pa_m\Gamma$. 

Every positive harmonic function has a representation as an integral of minimal harmonic functions with respect to some measure on the Martin boundary. In other words, given a positive harmonic function $f$, there exists a unique measure $\nu_f$ supported on the minimal Martin boundary such that for all $x \in \Gamma$, we have
$$
f(x)=\int_{\pa_m\Gamma}K_\mu(x,\xi)d\nu_f(\xi).
$$
One can see that if $f$ itself is minimal harmonic, then $\nu_f$ is a point measure, and hence $f=K_\mu(.,\eta)$ for some $ \eta$ in the boundary.


\section{Transformed random walks}\label{section: stopping time}
Consider a generating probability measure $\mu$ on a discrete group $\Gamma$.
Let $\tau:(\Gamma ^{\Bbb Z+},\p)\mapsto \Bbb N\cup\{\infty\}$ be a stopping time defined on the space of sample paths, that is, $\tau$ is a measurable map and for each integer $s \ge 1$ the set defined by $\tau^{-1}(s)$ is in the $\sigma$-algebra generated by the position of the random walk $(\Gamma,\mu)$ from time 0 to $s$. All stopping times considered in this paper are assumed to be finite almost everywhere.

Since random walks on groups are time and space homogeneous, any stopping time $\tau$ can 
be iterated to construct a new random walk\cite{BK2013}. More formally, set 
$\tau_1=\tau$, and for $n \ge 1$ define $(n+1)$-th iteration of $\tau$ inductively via
\begin{equation}\label{iteration}
\tau_{n+1}=\tau_n+\tau(U^{\tau}),
\end{equation}
where
\begin{equation}\label{induced}
U(\x)=(x_1^{-1}x_{n+1})_{n\geq0}
\end{equation}
is the transformation of the path space induced by the time shift in the space of increments.

\medskip
The subsequence $(x_{\tau_n})$ can now be regarded as a sample path of the random walk $(\Gamma ,\mu_{\tau})$, where
$\mu_{\tau}$ is  the distribution of  $x_{\tau}$, i.e.,
$$
\mu_{\tau}(g)=\p\{\x\ :\ x_{\tau}=g\}.
$$
Obviously, each $\tau_n$ is also a  stopping time, and, moreover,
the distribution of $x_{\tau_n}$ is the $n$-fold convolution of $\mu_{\tau}$, i.e, $\mu_{\tau_{n}}=(\mu_{\tau})^{*n}$.
Let us give some examples of this construction:
\begin{ex}\label{convolution}
Let $k$ be a positive integer, and let  $\tau$ be the constant function $k$. Then $\mu_{\tau}$ is the $k$-fold convolution of measure $\mu$,
i.e., $\mu^{*k}$.
\end{ex}

\begin{ex}
Let $\tau_1$ and $\tau_2$ be two  stopping times for the random walk $(\Gamma ,\mu)$. Then,
$\tau=\tau_1+\tau_2(U^{\tau_1})$ is a stopping time, 
and $\mu_{\tau}=\mu_{\tau_1}*\mu_{\tau_2}$.
\end{ex}

\begin{ex}\label{wil}
Let $B \subseteq \Gamma$ with $\mu(B)>0$.  For a sample path $\x=(x_n)$, let 
$$
\tau(\boldsymbol{x})=\min\{i\geq1\ :\ h_i\in B\},
$$
denote the least value of $i$ for which the increment $h_i$ lands in $B$. It can be easily seen that $\tau$ is a  stopping time. If $\mu(B)=1$, then trivially $\mu_{\tau}=\mu$; otherwise
$$
\mu_{\tau}= \sum_{i=0}^{\infty}\alpha^{*i}*\beta,
$$
where $\beta$ is the restriction of $\mu$ into $B$, i.e.,  $\beta(A)=\mu(A\cap B)$ for a subset $A$ of $\Gamma$, and $\alpha=\mu-\beta$.
\end{ex}
\subsection{Randomized stopping time}
In this subsection we will recall the definition of randomized stopping times. Readers are referred for more details to \cite{BK2013,Tsirelson}. 

Let $(\Omega,\lambda)$ be a probability space. A measurable function $\tau: \Gamma^{\Z+}\times \Omega\to\Bbb N\cup\{\infty\}$ is called a \emph{randomized stopping time} whenever for every $\omega$ in $\Omega$, the measurable function $\tau(.,\omega)$ is a stopping time. We will always assume that $\T$ is almost everywhere finite and set $\mu_\T(g)=\p\otimes\lambda\{(\x,\omega): \x_{\tau(\x,\omega)}=g\}$.

\begin{ex}
By choosing $\Omega$ to be a singleton, each stopping time $\tau$ can be viewed as a (deterministic) randomized stopping time. 
\end{ex}
\begin{ex}\label{ex : convex}
Let $\rho_n:\Gamma^{\Z+}\times\Omega_n\to\Bbb N\cup\{\infty\}$ be a randomized stopping time for $n\geq 1$. Let  $(a_i)_{i\geq1}$ be a nonnegative sequence such that $\sum_{i\geq 1} a_i=1$. For $(\x,\omega_1,\omega_2,\cdots,k)\in \Gamma ^{\Z+}\times\Omega_1\times\Omega_2\times\cdots\times\Bbb N$, define 
$\T'(\x,\omega_1,\omega_2,\cdots,k)=\rho_k(\x,\omega_k)$. Then $\T'$ is a randomized stopping time and $\mu_{\T'}=\sum_na_n\mu_{\rho_n}$.
As a special case, when $\rho_n=n$, we obtain $\mu'=\sum_na_n\mu^{*n}$.  
\end{ex}
\begin{ex}
Let $\T$ be a randomized stopping time and $k$ be a natural number.  Then $\rho=\min\{\T,k\}$ is a bounded randomized stopping time. 
\end{ex}
\section{The Martin boundary of free semigroups} 
In this section, we will describe the Martin boundary of a free semigroup under some special random walks. 
\subsection{The geometric boundary of a free semigroup}
Let $\F$ denote the free semigroup on the free generating (finite or infinite countable) set $W$ whose cardinality will be referred to as the rank of $\F$. Note that each element of $\F$ can be uniquely expressed as a finite word $x=w_1 \dots w_n$ with $w_i \in W$, 
$1 \le i \le n$, in this case, $|x|:=n$. The neutral (identity) element of $\F$, which can be viewed as the empty word on $W$, will be denoted by $e$. One can naturally construct a directed graph with the vertex set $\F$ as follows: for $x$ and $w$ in $\F$, there is a directed edge from $x$ to $xw$ for each $w$ in $W$.
It is easy to see that the graph obtained in this way is a tree in which every vertex has the in-degree one (except for $e$, for which the in-degree
is zero) and the out-degree the cardinality of $W$. For $x$ and $y$ in $\F$, we write $x\leq y$  when there exists $h$  in $\F$ such that $y=xh$. When this holds, the element 
$h$ will be denoted by $x^{-1}y$. Geometrically, $x \le y$ is equivalent to the condition that $x$ lies on the geodesic between $e$ and $y$ on the Cayley graph of $\F$.

\medskip
Denote the set of all infinite words associated to $\F$ by $\F^{\infty}$.  Let $\F^*=\{y^*\ :\ y\in\F\}$ be a copy of $\F$ which is set-theoretically disjoint from $\F$.  Set

$$\hat\F=\Bigg\{\begin{array}{lr}
\F\sqcup\F^{\infty}, & \F \mbox{ is finite rank} \\
 & \\
\F\sqcup\F^*\sqcup\F^{\infty}, &   \F \mbox{ is infinite rank}.
\end{array}
$$

For each finite word $x$ in $\F$, let $\F_x:=x\F$ be the set of all finite words starting with $x$. Similar to $\hat\F$, define $\hat{\F}_x= x\F \sqcup x^{\ast} \F^*\sqcup x\F^{\infty} $.  For a finite word $x$ in $\F$ and a finite subset  $E \subseteq  W$, let 
$E(x)=\bigcap_{w\in E}\hat{\F}\backslash\hat{\F}_{xw}$. The \emph{geometric boundary} of $\F$ is  $\hat\F$  which is equipped with the topology generated by
$$
\{\{x\},\hat{\F}_x, E(x)\ :\ x\in\F\}.
$$
The geometric boundary of $\F$ is indeed a compactification of $\F$, see \cite{Woess09}. Let us briefly remark that if $\F$ has finite rank (finitely generated), then from any infinite sequence in $\F$ one can extract an infinite subsequence whose terms are initial segments of an infinite word. This is no longer the case when $\F$ has infinite rank (infinitely generated). For example, for any sequence $w_1,w_2,\cdots$ of distinct generators of $\F$, no subsequence of the sequence $(w_n)_{n\geq1}$ has a limit in $\F\sqcup\F^{\infty}$. In fact,  $\lim_nw_n= e^*$ in the geometric boundary $\hat\F$.  More generally, it also is easy to see a sequence
$(x_n)_{n \ge 1}$ in $\F$ converges to a point in the geometric boundary of $\F$ if and only if it falls under one of the following three classes:
\begin{enumerate}
\item Eventually constant sequence: $x_n\to x\in\F$ if except for finitely many values of $n$, we have $x_n=x$.
\item Sequences with increasing initial common segments: here $x_n\to \g\in\F^{\infty}$,  where except for finitely many $n$, we have $x_n< \g$ and $\lim_n|x_n|=\infty$.
\item Sequences with a finite common initial segment: $x_n \to x^*\in\F^*$ if except finitely many values of $n$, the finite word $x_n$ is in $\F_x$  and for every finite subset $E$ of the generating set $W$ we have $\{ x_n: n \ge 1 \} \not\subseteq \bigcup_{w\in E}\F_{xw}$. 
\end{enumerate} 

Note that (3) only happens  when $\F$ is  infinitely generated. 

\subsection{Semi-deterministic Random walk on $\F$}
The goal of this section is to show that the Martin boundary of some special random walks on $\F$ can be identified with the geometric boundary of $\F$. Our approach is similar to
the one used for identifying the Martin boundary of the nearest neighbor random walks on trees discovered by Cartwright, Soardi, and Woess \cite{Carwright-Soardi-Woess-1993}.

We will start by endowing $\F$ with a probability measure $\mu$ distributed on the generating set $W$, and consider the random walk $(\F,\mu)$ which is a Markov chain issued from identity element (empty word) $e$ of $\F$ with transition probabilities given by
$$
p(x,xw)=\mu(w)
$$
for all $x$ in $\F$ and $w$ in $W$.
For any natural number $n$, define $\mu^{n}(x):=\mu(w_1)\cdots\mu(w_n)$ where $x=w_1\cdots w_n$ for some $w_i$ in $W$ and $i=1,\cdots,n$. As there is no cancellation in the random walk $(\F,\mu)$ and $\mu$ is distributed on the generating set $W$, the $n$-fold convolution of $\mu$  
is concentrated on all the words whose distance from identity is $n$. Therefore, we have $\mu^{*n}(x)=\mu^n(x)$. 
Consequently, we have the following explicit description of the Green functions: 
$$
G_{\mu}(x,y)=
\left\{
	\begin{array}{ll}
		\mu^{|x^{-1}y|}(x^{-1}y)  & \mbox{if } x \leq y \\
		0 & \mbox{otherwise }
	\end{array}
\right.
$$
Consequently, we have
$$
K_{\mu}(x,y)=\frac{G_{\mu}(x,y)}{G_{\mu}(e,y)}=
\left\{
	\begin{array}{ll}
\frac{1}{\mu^{|x|}(x)} & \mbox{if } x \leq y \\
		0 & \mbox{otherwise }
	\end{array}
\right.
$$

Assume that $\tau$ is a stopping time for the random walk $(\F,\mu)$. Let
$\F_\T$ be the  semigroup generated by the support of $\mu_\T$, which is indeed a free semigroup.  Since there is no backtracking in the random walk $(\F,\mu)$, the element $x$ from $\F$ is in the support of $\mu_\T$ if and only if $\mu_\T(x)=\mu^{|x|}(x)$. As with the Green function of the random walk $(\F,\mu)$, for $x$ and $y$ in $\F_\T$, we have 
$$
G_{\mu_\T}(x,y)=
\left\{
	\begin{array}{ll}
		\mu^{|x^{-1}y|}(x^{-1}y)  & \mbox{if } x \leq y \\
		0 & \mbox{otherwise }
	\end{array}
\right.
\hspace{.6cm}
K_{\mu_\T}(x,y)=
\left\{
	\begin{array}{ll}
\frac{1}{\mu^{|x|}(x)} & \mbox{if } x \leq y \\
		0 & \mbox{otherwise }
	\end{array}
\right.
$$
It is easy to extend these Martin kernels to all the points in $\hat{\F}$ by setting 
$$
K_{\mu_{\T}}(x,\g)= \left\{
	\begin{array}{ll}
\frac{1}{\mu^{|x|}(x)} & \mbox{if } x \leq \g \\
		0 & \mbox{otherwise }
	\end{array}
\right.
$$
where $\g\in\F_\T^{\infty}$ and  $K_{\mu_{\T}}(x,y^*)=K_{\mu_\T}(x,y)$ whenever $y^*\in\F_\T^*$  (when  $\F_\T$ is infinitely generated).  
Let $(x_n)_{n\geq1}$ be a sequence in $\F$ which converges to boundary point $\xi$. Since $K_{\mu_\T}(x,.)\in\{0,\frac{1}{\mu^{|x|}(x)}\}$,  the characterization of convergence in the boundary points of $\hat\F$ implies that the sequence $\Big(K_{\mu_\T}(x,x_n)\Big)_{n\geq 1}$ is eventually constant and equal to $K_{\mu_\T}(x,\xi)$. Consequently, the extensions of Martin kernels to the boundary points is continuous. We also have $K_{\mu_\T}(.,\xi)\not=K_{\mu_\T}(.,\eta)$  for $\xi\not=\eta$. We thus have proven the following theorem: 
\begin{theorem}
The Martin boundary of the random walk $(\F_\T,\mu_\T)$ is homeomorphic to the geometric boundary of $\F_\T$. 
\end{theorem} 
We can now characterize the minimal harmonic functions with respect to the random walk $(\F_\T,\mu_\T)$.
\begin{theorem}\label{minimal}
For any stopping time $\tau$, the set $\{K_{\mu_\T}(.,\g) :\ \g\in\F_\T^{\infty}\ \}$ coincides with the minimal $\mu_\T$--harmonic functions.
\end{theorem}
\begin{proof}
It is  easy to see that $K_{\mu_{\T}}(.,\g)$ is a $\mu_{\tau}$--harmonic function. Assume there exist two positive $\mu_{\tau}$--harmonic functions $f_1$ and $f_2$ such that $f_1(e)=f_2(e)=1$ and $0<a<1$ such that
$$
K_{\mu_{\T}}(.,\g)=af_1+(1-a)f_2\;.
$$
If $x$ does not lie on the geodesic $\g$, we have $K_{\mu_\T}(x,\g)=f_1(x)=f_2(x)=0$.  Since $f_i$ is $\mu_\T$--harmonic for $i=1,2$, we have
$$
f_i(x)=\sum_hf_i(xh)\mu_{\tau}(h)\;,
$$
and consequently, $f_i(x)\geq f_i(xh)\mu_{\tau}(h)$. Thus for any $h$ in the support of $\mu_{\T}$ such that $xh$ lies on the geodesic $\g$, we can write
$$
K_{\mu_{\T}}(x,\g)\geq K_{\mu_{\T}}(xh,\g)\mu_{\tau}(h)=K_{\mu_{\T}}(x,\g)\;,
$$
which implies that $f_i(x)=f_i(xh)\mu_{\tau}(h)$. Since $f_i(e)=1$, by induction, we have $f_1=f_2=K_{\mu_\tau}(.,\g)$, hence $K_{\mu_\tau}(.,\g)$ is a minimal harmonic function.  On the other hand, if $u$ is a minimal $\mu_\T$--harmonic, by the Poisson integral formula, there is $\xi$ in the Martin boundary such that $u=K_{\mu_\T}(.,\xi)$. Since $K_{\mu_\T}(.,y)$ is not a $\mu_\T$--harmonic function at $y$, therefore $\xi=\g$ for an infinite word. So  $u=K_{\mu_\T}(.,\g)$.
\end{proof}

If $\T$ is a bounded stopping time, then $\F_\T^{\infty}=\F^{\infty}$. 
When $\T$ is not  bounded, there can exist a subset of sample paths with measure zero over which the stopping time takes the value infinity. Although this is a null set, it can change the Martin boundary. In fact, the infinite sequences with respect to these sample paths do not belong to $\F_\T^\infty$. Therefore, the minimal boundary with respect to $(\F_\T,\mu_\T)$ may be different from that of $(\F,\mu)$. 
For example, consider the simple random walk on $\F_3$, the free semigroup generated by three generators $a,b$, and $c$. Let $\tau$ be the first time that an increment equals $a$. Then any infinite word consisting only of letters $b$ and $c$ does not belong to $\F_\T^\infty$.

Note that for every infinite word $\g$ in $\F_\T$, the minimal $\mu_\T$--harmonic functions $K_{\mu_\T}(.,\g)$ can be uniquely extended to the minimal $\mu$--harmonic function $K_{\mu}(.,\g)$. Hence, it follows
\begin{cor}
Let $\T$ be a  bounded stopping time. Then, a positive $\mu$--harmonic function restricts to a positive
$\mu_\T$--harmonic function. Moreover, each $\mu_\T$--harmonic function extends uniquely to a $\mu$--harmonic function.
\end{cor}

\subsection{Randomized stopping measure $\mu_{\T,t}$}
Again let $\mu$ be a probability measure supported on the generating set $W$. Fix a real number $0<t<1$, and consider the convex combination $\mu_{\T,t}=t\mu+(1-t)\mu_\T$, where $\T$ is a  bounded randomized stopping time. In this section, we will characterize all minimal $\mu_{\T,t}$--harmonic functions. 

First note that by Example~\ref{ex : convex}, there exists a randomized stopping time $\T'$ such that 
$\mu_{\T'}=\mu_{\T,t}$, more precisely,  let $\T'$ be the randomized stopping time which equals $1$ with probability $t$ 
and equals $\tau$ with probability $1-t$.


 Since $\mu$ is distributed on $W$, for any $x \in \F$, and any $n \ge 1$ there exists $c_n(x)$ in the interval $[0,1]$ such that $\mu^{*n}_{\T,t}(x)=c_n(x)\mu^{|x|}(x)$. Consequently,
\begin{equation}\label{eq:green T,t}
G_{\mu_{\T,t}}(e,x)=\sum_{n=0}^{|x|}\mu^{*n}_{\tau, t} (x)=c(x)\mu^{|x|}(x)
\end{equation}
where $c(x)=\sum_{n=0}^{|x|}c_n(x)$.

We are going to show that any positive minimal $\mu_{\T,t}$--harmonic function can be identified with an infinite word.
\begin{theorem}\label{thm : geodesic minimal} 
If the positive function $u$ is minimal $\mu_{\T,t}$--harmonic, then there exists an infinite word $\g$ such that $u(x)>0$ if and only if $x<\g$.
\end{theorem}
\begin{proof}
Assume that $u(e)=1$. Since $u$ is $\mu_{\T,t}$--harmonic, if it vanishes on the generating set $W$, it will be identically zero. Choose $g_1\in W$ such that $u(g_1)>0$, and define 
$$
u_1(x)=\left\{
	\begin{array}{ll}
		u(x)  & \mbox{if } g_1 \leq x \\
		k_1 & x=e\\
		0 & \mbox{otherwise }
	\end{array}
\right.
$$
where $k_1=\sum_{y}u(g_1y)\mu_{\T,t}(g_1y)$. It is evident that $u_1 \le u$ and that $u_1$ is $\mu_{\T,t}$--harmonic. Since $u$ is minimal, it follows that there exists $c_1>0$ such that $c_1u_1=u$. Since $u_1(g_1)=u(g_1)$, we conclude that $u_1=u$. 
The argument can be continued by induction. More precisely, set $g_0=e$ and assume by induction that $g_1, g_2,\cdots,g_n \in W$ are such that $u(x)$ is strictly positive whenever $g_1\cdots g_n\leq x$ or $x=g_0\cdots g_i$ for some $ 0\leq i\leq n$.  Therefore,  we have $0<u(g_1\cdots g_n)=\sum_{g}u(g_1\cdots g_ng)\mu_{\T,t}(g)$, hence there exists $g_{n+1}\in W$ such that $u(g_1\cdots g_ng_{n+1})\mu(g_{n+1})>0$. Define 
  $$
u_{n+1}(x)=\left\{
	\begin{array}{ll}
		u(x)  & \mbox{if } g_1\cdots g_{n+1} \leq x\\
		k_{i+1} & x=g_1\cdots g_i \mbox{ for some }i=0,\cdots n\\
		0 & \mbox{otherwise }
	\end{array}
\right.
$$
$k_i$ can be chosen such that $u_{n+1}$ becomes a $\mu_{\T,t}$--harmonic and $u_{n+1}\leq u$. Hence, $u=u_{n+1}$. Therefore, there exists an infinite word $\g=(g_n)_{n\geq1}$ such that $u>0$ only on the words which lie on $\g$.
\end{proof}

For a fixed infinite word $\g=(g_n)_{n\geq1}$,  let $u^{\g}$ be the minimal $\mu$--harmonic function associated to $\g$, which means 
$$
u^{\g}(x)=K_{\mu}(x,\g)=\left\{
	\begin{array}{ll}
		\frac{1}{\mu^{|x|}(x)}  & \mbox{if } x \leq \g \\
		0 & \mbox{otherwise }
	\end{array}
\right.
$$
A simple calculation shows $u^{\g}$ is a $\mu_{\T,t}$--harmonic. We are going to show that $u^{\g}$ is also minimal with respect to the random walk $\mu_{\T,t}$.

\begin{lem}\label{lem:Z markov}
Consider a Markov chain on $\Z_+$ with transition function $q$ which satisfies the following two properties:
\begin{enumerate}
\item There exists a positive integer $M$  such that $\sum_{i=1}^Mq(n,n+i)=1$  for every $n \ge 0$.
\item There exists $t>0$ such that for all $n \ge 0$, we have $q(n,n+1)\geq t$
\end{enumerate}
Then every positive $q$-harmonic function is constant.
\end{lem}
\begin{proof}
Let $f$ be a positive $q$--harmonic function with $f(0)=1$. First, we claim that $t^M\leq f\leq t^{-M}$. This, in particular, shows that when $f$ is a minimal $q$--harmonic function, then $f=ct^M$ for some $c>0$, hence every minimal $q$--harmonic function is constant.  Consequently, by the Poisson representation formula, every positive harmonic functions is constant. 

Since $f>0$ is $q$--harmonic, we have
$$
f(n)=\sum_{i=1}^{M}f(n+i)q(n,n+i),
$$
and $f(0)=1$. We will  show that there exists a strictly increasing sequence $(a_n)_{n\geq1}$ such that $a_{n+1}-a_{n} \le M$ and $f(a_n) \ge 1$ for all $n \ge 1$. We will construct the 
sequence inductively. Set $a_1=0$, and suppose $a_1< \dots < a_m$ have been constructed. 
Since $f$ is $q$--harmonic and $f(a_m)\geq1$, there should exists $1\leq i\leq M$ such that $f(a_m+i)\geq 1$. Set $
a_{m+1}=a_m +i$ and obviously $a_{m+1}-a_m \le M$. 
We will now prove that $f(n) \geq t^M$ for all $n \ge 0$.  Thanks to the construction of the sequence $a_m$, for any $n \ge 1$, there exists $m$ such that $ a_m \leq n < a_{m}+M$.
This implies that  there must exist $1 \le j\le  M$, such that $n+j=a_{m+1}$. Note that for every $k$, we have $f(k)\geq t f(k+1)$. Applying this to $n$ repeatedly, we have 
\[ f(n) \geq t^{j} f(n+j)=t^{j} f(a_{m+1})\geq t^M. \]
Mimicking the same argument, we can show that $f\leq t^{-M}$. Hence, the claim follows.  

\end{proof}

\begin{theorem}\label{thm:min}
Let $\mu$ be a probability measure distributed on the generators of a free semigroup. 
If $\T$ is a bounded randomized stopping time, then a positive $\mu_{\T,t}$--harmonic function $u$ is minimal if and only if there exists an infinite word $\g$ such that $u=u^{\g}$. 
\end{theorem}
\begin{proof}
Fix an infinite word $\g$. First, we will show that $u^{\g}$ is  minimal. It suffices 
to prove that the bounded harmonic functions with respect to the conditional random walk 
$$p^{\g}(x,xy)=\mu_{\T,t}(y)\frac{u^{\g}(xy)}{u^{\g}(x)}=\frac{\mu_{\T,t}(y)}{\mu^{|y|}(y)}$$ for $x<\g$ and $xy<\g$ are constant. Let $M$ be the maximum of the randomized stopping $\T$. Note that the conditional random walk $p^{\g}$ can be identified as a Markov chain on $\Z+$ which satisfies the conditions in Lemma~\ref{lem:Z markov}. This implies that $u^{\g}$ is a minimal $\mu_{\T,t}$--harmonic function. 

Now, let $u$  be a minimal $\mu_{\T,t}$--harmonic function. There should exist $\xi$ in the Martin boundary 
such that $u=K_{\mu_{\T,t}}(.,\xi)$.  Define 
$$
h(x,\xi)=K_{\mu_{\T,t}}(x,\xi)\mu^{|x|}(x).
$$ 
Since $K_{\mu_{\T,t}}(.,\xi)$ is minimal, by Theorem~\ref{thm : geodesic minimal}, there exists an infinite word $\g$ such that $h(x,\xi)>0$ if and only if $x<\g$. Because $K_{\mu_{\T,t}}(.,\xi)$ is $\mu_{\T,t}$--harmonic, we have $h(x,\xi)$ is a positive harmonic function with respect to the conditional random walk $p^{\g}$, by the first part, $h(x,\xi)$ should be constant for any $x<\g$. We have $h(x,\xi)=h(e,\xi)=K_{\mu_{\T,t}}(e,\g)=1$ for $x<\g$, consequently, 
$K_{\mu_{\T,t}}(.,\xi)=u^{\g}$. 
\end{proof}

\begin{cor}\label{cor: freesemigroup}
Let $\mu$ be a probability measure distributed on the generators of a free semigroup. If $\T$ is a  bounded randomized stopping time and $\mu_{\T,t}=t\mu+(1-t)\mu_\T$ for some $0<t\leq1$, then $H^+(\F,\mu)=H^+(\F,\mu_{\T,t})$.
\end{cor}
Using the proof of Theorem~\ref{thm:min}, we can describe the Martin boundary of the random walk $(\F,\mu_{\T,t})$:
\begin{cor}
Let $\mu$ and $\T$ be the same as the preceding theorem. Then, the Martin boundary of the random walk $(\F,\mu_{\T,t})$ is homeomorphic to the geometric boundary of $\F$.
\end{cor}
\section{Discrete groups}
This section is devoted to prove the main results.

\subsection{Harmonic functions of discrete groups} The approach of this subsection is similar to the one applied by Kaimanovich and the first author in \cite{BK2013}. 

The following lemma is analogous to the result of Dynkin--Maljutov~\cite{Dynkin-Maljutov61}. Their result is in the context of bounded harmonic functions associated to random walks on discrete groups. However, it can be generalized to positive harmonic functions associated to random walks on semigroups. 

Let $\phi$ be a semigroup homomorphism from a free semigroup $\F$ to a discrete semigroup or group $\Gamma$. A function $\hat{f}$ on $\F$ is called $\phi$--invariant if $\hat{f}(\hat{g}_1)=\hat{f}(\hat{g}_2)$ whenever $\phi(\hat{g}_1)=\phi(\hat{g}_2)$.
\begin{lem}\label{lem:Dynkin-Maljutov}
Consider a surjective semigroup homomorphism  $\phi:\F\to\Gamma$. Let $\hat{\mu}$ and $\mu$, the push--forward of $\hat{\mu}$ under $\phi$, be generating probability measures on $\F$ and $\Gamma$, respectively. Then the map
$$
\begin{array}{c}
H^+(\Gamma,\mu)\to H^+(\F,\hat{\mu})\\
f\mapsto f\circ\phi
\end{array}
$$
is a bijection from  $H^+(\Gamma,\mu)$ onto the space of $\phi$--invariant functions in $H^+(\F,\hat{\mu})$.  
\end{lem}
\begin{proof}
Let $\hat{f}$ be  $\phi$--invariant. Obviously, we can define the function $f$ on $\Gamma$ by $f(g)=\hat{f}(\hat{g})$ whenever $\phi(\hat{g})=g$. Assume $\hat{f}$ is $\hat{\mu}$--harmonic. 
For any $x$ in $\Gamma$, we have $\mu(x)=\sum_{\hat{y}\in\phi^{-1}(x)}\hat{\mu}(\hat{y})$.  Hence, 

\begin{equation}\label{eq:1 phi-invariant}
f(gx)\mu(x)=\sum_{\hat{y}\in\phi^{-1}(x)}f(gx)\hat{\mu}(\hat{y})=\sum_{\hat{y}\in\phi^{-1}(x)}f(g\phi(\hat{y}))\hat{\mu}(\hat{y})
\end{equation}
Let $\phi(\hat{g})=g$. Because $\phi$ is a semigroup homomorphism, we can write 
$f(g\phi(\hat{y}))=f(\phi(\hat{g})\phi(\hat{y}))=f(\phi(\hat{g}\hat{y}))=\hat{f}(\hat{g}\hat{y})$, combining it with Equation~\ref{eq:1 phi-invariant} yields to
$$
f(gx)\mu(x)=\sum_{\hat{y}\in\phi^{-1}(x)}\hat{f}(\hat{g}\hat{y})\hat{\mu}(\hat{y}).
$$
Therefore taking summation over all $x$ from both sides of the preceding equation leads us to
$$
\sum_{x\in \Gamma}f(gx)\mu(x)=\sum_{x\in \Gamma}\sum_{\hat{y}\in\phi^{-1}(x)}\hat{f}(\hat{g}\hat{y})\hat{\mu}(\hat{y})=\sum_{\hat{y}\in\F}\hat{f}(\hat{g}\hat{y})\hat{\mu}(\hat{y}).
$$
Since $\hat{f}$ is  $\hat{\mu}$--harmonic, the right hand side is equal to $\hat{f}(\hat{g})=f(g)$. Hence $f$ is $\mu$--harmonic.  Similarly, if $f$ is positive $\mu$--harmonic, then the positive function $\hat{f}=f\circ\phi$ is  $\phi$--invariant and $\hat{\mu}$--harmonic.
\end{proof}

\begin{theorem}\label{iso}
Let $\mu$ be a generating probability measure on a discrete group $\Gamma$, and let $\T$ be a bounded randomized stopping time. Then the space of positive harmonic functions associated with $\mu$  and $\mu_{\T,t}$ coincide.
\end{theorem}
\begin{proof}
We will show that the harmonic functions on $\Gamma$ arise from harmonic functions on a free semigroup with certain invariance properties. In order to do this, we will use the probability measure $\mu$ on $\Gamma$ to define a probability measure on the free semigroup $\F$ generated freely on the support of $\mu$. More precisely, each element of the support of $\mu$ (viewed merely as a set) is considered as a free independent generator of $\F$. To avoid confusion, we use the notation $ \hat{h}$ for the generator of $\F$ corresponding to $h \in \Gamma $, and set 
$W= \{ \hat{h}: h \in \supp \mu \}$.  Now, define the measure $\m$ by $\m( \hat{h}):=\mu(h)$ for each generator $\hat{h}$ of $\F$.
We can naturally define the map $\phi: \F \to \Gamma$ by $\phi(\hat{g})=g$. In other words, each element of $\F$ is precisely a positive word in $W$, and $\phi(\hat{g})$ is the evaluation of this formal product in group $\Gamma$.  It is easy to verify that $\phi$ is a
surjective semigroup homomorphism, moreover, the image of the probability measure $\m$ is $\mu$.

Assume now that $\T$ is a bounded randomized stopping time for the random walk $(\Gamma ,\mu)$. Since the increment spaces of the random walks $(\F,\m)$ and $(\Gamma,\mu)$ are isomorphic, the randomized stopping time $\T$ can provide us a randomized stopping time with respect to the random walk $(\F,\m)$.
Define the corresponding randomized stopping time $\ta((\hat{x}_n)_{n\geq0},\omega)=\T((\phi(\hat{x}_n))_{n\geq0},\omega)$ for the random walk $(\F,\m)$. By  virtue of Lemma~\ref{lem:Dynkin-Maljutov} and  Corollary~\ref{cor: freesemigroup}, we conclude that $f$ belongs to $H^+(\Gamma ,\mu)$ if and only if  $f$ is in $H^+(\Gamma ,\mu_{\T,t})$. 

\end{proof}
\subsection{Harmonic functions and convex combinations}
By virtue of Theorem~\ref{iso} and Example~\ref{ex : convex}, we can conclude that  if $\mu'$ is a finite convex combination of $\mu$ and $\mu_{\rho_i}$ where $\rho_i$ is a  randomized stopping time for $i=1,\cdots,n$, i.e., 
$$\mu'=a_1\mu+a_2\mu_{\rho_1}+\cdots +a_{n+1}\mu_{\rho_n}$$ 
and $1=a_1+\cdots+a_{n+1}$, then $H^+(\Gamma ,\mu)=H^+(\Gamma ,\mu')$.

We will show that we can relax the condition of finite combinations to an infinite one whenever the probability measures commute under convolution. 
\begin{theorem}\label{convex}
Consider a set of probability measures ${\theta_i}$ on a discrete group $\Gamma$ such that $\theta_1*\theta_i=\theta_i*\theta_1$ for all $i\geq1$. Let  $\theta =\sum_{i}a_i\theta_i$ where $a_i$ is nonnegative for all $i\geq 1$ and $\sum_{i\geq 1} a_i=1$. If $\theta_1$ is generating and $a_1>0$ such that $H^+(\Gamma,\theta_1)\subset H^+(\Gamma,\theta_i)$ for $i\geq 1$, then
$H^+(\Gamma ,\theta)= H^+(\Gamma ,\theta_1)$. 
\end{theorem}
\begin{proof}
The fact that the probability measure $\theta_1$  is generating and $a_1>0$ implies that $\theta$ is also generating.  Assume that $u$ is a minimal $\theta$--harmonic function. We have 
$$
u=P^{\theta}u=\sum_ia_iP^{\theta_{i}}u.
$$
Since the probability measure $\theta_1$ commutes with $\theta_i$ for all $i\geq1$, we deduce that $\theta_1$ commutes with $\theta$ too, and
$$
P^\theta P^{\theta_1}u=P^{\theta_1}P^\theta u=P^{\theta_1} u,
$$
which means that $P^{\theta_1}u$ is a $\theta$--harmonic function. Since $u$ is minimal, we have $P^{\theta_{1}}u=u$. It follows that each minimal $\theta$--harmonic function is $\theta_1$--harmonic, implying $H^+(\Gamma ,\theta)\subset H^+(\Gamma ,\theta_1)$. The reverse inclusion is clear.
\end{proof}

\begin{cor}
Let $\mu$ be a generating probability measure on a discrete group $\Gamma$.
Let $\mu'=\sum_{i\geq 1}a_i\mu^{*i}$ be a convex combination of convolutions of the probability measure $\mu$. 
If $a_1>0$, then $H^+(\Gamma,\mu)=H^+(\Gamma,\mu')$.
\end{cor}

\subsection*{Acknowledgment} 
The authors wish to use the opportunity to thank Vadim Kaimanovich and 
Wolfgang Woess for numerous valuable suggestions on an earlier version of this paper. We also thank S\'ebastian Gouz\"el and Iddo Ben-Ari for fruitful discussions. During the completion of this project, the second author was partially supported by the DFG grant DI506/14-1. 

\bibliographystyle{alpha}
\bibliography{Martin-bnd}
\end{document}